\documentclass[reqno,b5paper]{amsart}

\usepackage{amsmath}
\usepackage{amssymb}
\usepackage{amsthm}
\usepackage{enumerate}
\usepackage[mathscr]{eucal}
\usepackage{eqlist}
\usepackage{graphicx}
\theoremstyle{plain}
\newtheorem{theorem}{Theorem}[section]

\theoremstyle{definition}
\newtheorem{definition}{Definition}[section]
\newtheorem{remark}{\textnormal{\textbf{Remark}}}

\theoremstyle{remark}
\newtheorem{example}{Example}

\numberwithin{equation}{section}

\begin{document}
\title[Iterated Function Systems with the Average Shadowing Property]%
{Iterated Function Systems with the Average Shadowing Property}
\author[Mehdi Fatehi Nia]%
{Mehdi Fatehi Nia}

\newcommand{\acr}{\newline\indent}

\address{\llap{*\,}Department of Mathematics\acr
Yazd University\acr
 P. O. Box 89195-741 Yazd\acr
  Iran }
\email{fatehiniam@yazd.ac.ir}



\begin{abstract}The average shadowing property is considered for set-valued dynamical systems, generated
by parameterized \emph{IFS}, which are uniformly contracting, or conjugacy, or products
of such ones. We also prove that if a continuous surjective IFS $\mathcal{F}$ on a
compact metric space $X$ has the average shadowing property, then
every point $x$ is chain recurrent. Moreover, we introduce some examples and investigate  the relationship between the original shadowing property and shadowing property for IFS. For example, this is proved that the Sierpinski IFS has the average shadowing property. Then we shows that there is an IFS $\mathcal{F}$ on  $S^{1}$ such that $\mathcal{F}$ does not satisfy the average shadowing property but every point $x$ in  $S^{1}$ is a chain recurrent.
\end{abstract}

\maketitle

\section{Introduction}
 The shadowing property of a dynamical system  is one of the most important notions in dynamical systems(see \cite{[KP],[KS]}). The notion of average shadowing property was introduced by Michael Blank in \cite{[MB]} in order to study chaotic dynamical
systems, which is a good tool to characterize Anosov diffeomorphisms. The average shadowing property was further studied by several authors,
with particular emphasis on connections with other notions known from topological dynamics, or more narrowly, shadowing
theory (e.g. see \cite{[RG1],[Nb],[PZ]}). \\Iterated function systems( IFS), are used for the construction of deterministic fractals and have found numerous applications, in
particular to image compression and image processing \cite{[B]}. Important notions in
dynamics like attractors, minimality, transitivity, and shadowing can be extended to
IFS (see \cite{[BV],[BVA],[GG],[GG1]}). The authors defined the shadowing property for a parameterized iterated function system and prove that if a parameterized IFS is uniformly contracting or expanding, then it has the shadowing property\cite{[GG]}. They also generalized the shadowing property on affine parameterized IFS and  prove the following theorem:
\begin{theorem}
Given a closed nonempty subset $A \subset \mathbb{C}$, situated strictly
inside or strictly outside the unit circle. For any closed disc centered at $0$ with radius $r > 1$ in $\mathbb{C}$ and
any subset $B \subset \mathbb{C} $ the parameterized IFS $ F = {\mathbb{C}; f_{a,b} | a\in A, b \in B}$, with
$f_{a,b}(z) = az + b$, has the Shadowing Property on $Z_{+}$.
\end{theorem}
 The present paper concerns the average shadowing property for parameterized IFS and some important results about average shadowing property are extended to iterated function systems. Specially, we prove that each uniformly contracting parameterized IFS has the average shadowing property on $\mathbb{Z}_{+}$. Then we introduce a condition such that if an IFS,  $\mathcal{F}$, on unit circle satisfies this condition then it does not have the average shadowing property. In Section $4$, we give some examples illustrating our results.
In \cite{[AH]} the authors have shown that a minimal homeomorphism on a compact metric space does not have the shadowing property. In Example \ref{e3} we show that there is a minimal IFS which has the shadowing property. Example \ref{e7} shows that there is an \emph{IFS}, $\mathcal{F}$, on  $S^{1}$ such that $\mathcal{F}$ does not satisfy the average shadowing property but every point $x$ in  $S^{1}$ is a chain recurrent.
\section{preliminaries}\par In \cite{[RG1],[Nb],[PZ]} the average shadowing property is defined and discussed for continuous
maps. Let $f:X\longrightarrow X$ be a continuous map. For $\delta>0$, a sequence $\{x_{i}\}_{i\geq 0}$ of points in $ X$ is called a $\delta-$average-pseudo-orbit of $f$ if there is a number $N=N(\delta)$ such that for all $n\geq N$ and $k\geq 0$,
 $$\frac{1}{n}\sum_{i=0}^{n-1}d(f(x_{i+k}),x_{i+k+1})<\delta.$$
 We say that $f$ has the average shadowing property if for every $\epsilon>0$ there is $\delta>0$ such that every $\delta-$average-pseudo-orbit
 $\{x_{i}\}_{i\geq 0}$ is $\epsilon-$shadowed in average by some point $y\in X$, that is,
 $$\limsup_{n\rightarrow\infty}\frac{1}{n}\sum_{i=0}^{n-1}d(f^{i}(y),x_{i})<\epsilon.$$
\par We use the notion of chain recurrent points to study chaotic dynamical systems.\par The set $CR(f)$ consisting of all chain recurrent pints, i.e., such points $x\in X$ that for any $\delta>0$, there exists a periodic $\delta-$pseudo-orbit through $x$, is called the chain recurrent set of the discrete dynamical system $(X,f).$
\\Let $(X,d)$ be a complete metric space. Let us recall that a \emph{parameterized Iterated Function System(IFS)} $\mathcal{F}=\{X; f_{\lambda}|\lambda\in\Lambda\}$ is any family of continuous mappings $f_{\lambda}:X\rightarrow X,~\lambda\in \Lambda$, where $\Lambda$ is a finite nonempty set (see\cite{[GG]}).\\ Let $T=\mathbb{Z}$ or $T=\mathbb{Z}_{+}= \{n\in \mathbb{Z}:n\geq0\}$ and $\Lambda^{T}$ denote the set
of all infinite sequences $\{\lambda_{i}\}_{i\in T}$ of symbols belonging to $\Lambda$. A typical element of $\Lambda^{\mathbb{Z}_{+}}$
 can be denoted as $\sigma= \{\lambda_{0},\lambda_{1},...\}$ and we use the shorted notation $$\mathcal{F}_{\sigma_{n}}=f_{\lambda_{n-1}}o ...o f_{\lambda_{1}} o f_{\lambda_{0}}.$$ A sequence $\{x_{n}\}_{n\in T}$ in $X$ is called an orbit of the IFS $\mathcal{F}$ if there exist $\sigma\in \Lambda^{T}$ such that $x_{n+1}=f_{\lambda_{n}}(x_{n})$, for each $\lambda_{n}\in \sigma$.\\ Given $\delta>0$, a sequence $\{x_{n}\}_{n\in T}$ in $X$ is called a $\delta-$pseudo orbit of $\mathcal{F}$ if there exist $\sigma\in\Lambda^{T}$ such that for every $\lambda_{n}\in \sigma$, we have \\$d(f_{\lambda_{n}}(x_{n}),x_{n+1})<\delta$.\\One says that the IFS $ \mathcal{F}$ has the \emph{shadowing property }(on $T$) if, given $\epsilon>0$, there exists $\delta>0$ such that for any $\delta-$pseudo orbit $\{x_{n}\}_{n\in T}$ there exist an orbit $\{y_{n}\}_{n\in T}$, satisfying the inequality $d(x_{n},y_{n})\leq \epsilon$ for all $n\in T$. In this case one says that the $\{y_{n}\}_{n\in T}$ or  $y_{0}$, $\epsilon-$ shadows the $\delta-$pseudo orbit $\{x_{n}\}_{n\in T}$.\\
 For $\delta > 0$, a sequence $\{x_{i}\}_{i\geq 0}$ of points in $X$ is called a
$\delta$-average pseudo-orbit of $ \mathcal{F}$ if there exists a natural number $N = N(\delta) > 0$ and\\  $\sigma=\{\lambda_{0},\lambda_{1},\lambda_{2},...\}$ in  $\Lambda^{\mathbb{Z}_{+}}$,
such that for all $n \geq N$,
$$\frac{1}{n}\sum_{i=0}^{n-1}d(f_{\lambda_{i}}(x_{i}),x_{i+1})<\delta.$$
A sequence $\{x_{i}\}_{i\geq 0}$ in $X$ is said to be $\epsilon-$shadowed in average by a point $z$ in $X$ if there exist $\sigma\in\Lambda^{\mathbb{Z}_{+}}$ such that
$$\limsup_{n\rightarrow\infty}\frac{1}{n}\sum_{i=0}^{n-1}d(\mathcal{F}_{\sigma_{i}}(z),x_{i})<\epsilon.$$ A continuous IFS $\mathcal{F}$ is said to have the  average shadowing property if for every $\epsilon>$ there is $\delta>0$ such that every $\delta-$average-pseudo-orbit of $f$ is $\epsilon-$shadowed in average by some point in $X$.\\
Please note that if $\Lambda$ is a set with one member then the parameterized IFS $ \mathcal{F}$ is an ordinary discrete dynamical system.
In this case the average shadowing property for $\mathcal{F}$ is ordinary average shadowing property for a discrete dynamical system.\\The parameterized IFS
$ \mathcal{F}=\{X; f_{\lambda}|\lambda\in\Lambda\}$ is \emph{uniformly contracting} if there exists
$$\beta= sup_{\lambda\in\Lambda} sup_{x\neq y}\frac{d(f_{\lambda}(x),f_{\lambda}(y))}{d(x,y)} $$and this number called also the \emph{contracting ratio},
 is less than one\cite{[GG]}.
  \begin{definition}
  Suppose $(X,d)$ and $(Y,d^{'})$ are compact metric spaces and $\Lambda$ is a finite set. Let $ \mathcal{F}=\{X; f_{\lambda}|\lambda\in\Lambda\}$ and $ \mathcal{G}=\{Y; g_{\lambda}|\lambda\in\Lambda\}$ are two $IFS$ which $f_{\lambda}:X\rightarrow X$ and  $g_{\lambda}:Y\rightarrow Y$ are continuous maps for all $\lambda\in\Lambda$. We say that $ \mathcal{F}$ is topologically conjugate to $ \mathcal{G}$  if there is a homeomorphism  $h:X\rightarrow Y$  such that $g_{\lambda}=h o f_{\lambda} o h^{-1}$, for all $\lambda\in\Lambda$. In this case, $h$ is called a topological conjugacy. \end{definition}
 \section{\textbf{ Average Shadowing Property for Iterated Function Systems}}
In this section we investigate the structure of parameterized \emph{IFS} with the average shadowing property.
 It is well known that if $f:X\rightarrow X$ and $g:Y\rightarrow Y$ are conjugated then $f$ has the shadowing property if and only if so does $g$. In the next theorem we extend this property for \emph{iterated function systems }.
\begin{theorem}
Suppose $(X,d)$ and $(Y,d^{'})$ are compact metric spaces and $\Lambda$ is a finite set. Let $ \mathcal{F}=\{X; f_{\lambda}|\lambda\in\Lambda\}$ and $ \mathcal{G}=\{Y; g_{\lambda}|\lambda\in\Lambda\}$ are two conjugated IFS with topological conjugacy $h$.  If there exist two positive real number, $K, ~L$ such that $L<\frac{d^{'}(h(p),h(q))}{d(p,q)}<K$, for all $p,q\in X$. Then $\mathcal{F}$ has the average shadowing property if and only if so does $\mathcal{G}$.
\end{theorem}

\begin{proof}
Given $\epsilon>0$.  Let $\delta>0$ be an $\frac{\epsilon}{K}-$modulus the average shadowing property $\mathcal{F}$. Suppose  $\{y_{i}\}_{i\geq 0}$ is a $L\delta$-average pseudo orbit of $\mathcal{G}$. This means that there exists a natural number $N = N(\delta) > 0$ and \\ $\sigma=\{\lambda_{0},\lambda_{1},\lambda_{2},...\}$ in  $\Lambda^{\mathbb{Z}_{+}}$,
such that for all $n \geq N$, $$\frac{1}{n}\sum_{i=0}^{n-1}d^{'}(g_{\lambda_{i}}(y_{i}),y_{i+1})<L\delta.$$ Put $x_{i}=h^{-1}(y_{i})$, for all $i\geq 0$. Then $$d(f_{\lambda_{i}}(x_{i}),x_{i+1})=d(h^{-1} o g_{\lambda_{i}}(y_{i}),h^{-1}(y_{i+1}))<\frac{d^{'}( g_{\lambda_{i}}(y_{i}),y_{i+1})}{L},$$ for all $i\geq 0$. Thus $$\frac{1}{n}\sum_{i=0}^{n-1}d(f_{\lambda_{i}}(x_{i}),x_{i+1}))<\frac{L\delta}{L}=\delta,$$
for all $n \geq N$. Hence $\{x_{i}\}_{i\geq 0}$ is a $\delta-$average pseudo orbit for $\mathcal{F}$ and so there is an orbit $\{z_{i}\}_{i\geq 0}$ of  $\mathcal{F}$ such that $$\limsup_{n\rightarrow\infty}\frac{1}{n}\sum_{i=0}^{n-1}d(z_{i},x_{i})<\frac{\epsilon}{K}.$$ Note that for each $i\geq 0$ there is $\mu_{i}\in \Lambda$ such that $z_{i+1}=f_{\mu_{i}}(z_{i})$. Let $w_{i}=h(z_{i})$, this is clear that \\ $w_{i+1}=h(z_{i+1})=h(f_{\mu_{i}}(z_{i}))=g_{\mu_{i}}(h(z_{i}))=g_{\mu_{i}}(w_{i})$ for all $i\geq 0$. Therefor $\{w_{i}\}_{i\geq 0}$ is an orbit of $\mathcal{G}$ and $d^{'}(w_{i},y_{i})=d^{'}(h(z_{i}),h(x_{i}))<K d(z_{i},x_{i})$ for all $i\geq 0$.  Then $$\limsup_{n\rightarrow\infty}\frac{1}{n}\sum_{i=0}^{n-1}d^{'}(w_{i},y_{i})=\limsup_{n\rightarrow\infty}\frac{1}{n}\sum_{i=0}^{n-1}d^{'}(h(z_{i}),h(x_{i}))<K\frac{\epsilon}{K}=\epsilon.$$
\end{proof}
\begin{theorem}\label{tc}
If a parameterized IFS $\mathcal{F}=\{X; f_{\lambda}|\lambda\in\Lambda\}$ is uniformly contracting, then it has the average shadowing property on $\mathbb{Z}_{+}$.
\end{theorem}
\begin{proof}
Assume that $\beta<1$ is the contracting ratio of $\mathcal{F}$. Given $\epsilon>0$, take $\delta=\frac{(1-\beta)\epsilon}{2}\leq\frac{\epsilon}{2}$ and suppose $\{x_{i}\}_{i\geq 0}$ is a $\delta-$pseudo orbit for $ \mathcal{F}$. So there there exists a natural number $N = N(\delta) $ and \\ $\sigma=\{\lambda_{0},\lambda_{1},\lambda_{2},...\}\in\Lambda^{\mathbb{Z}_{+}}$ such that $\frac{1}{n}\sum_{i=0}^{n-1}d(f_{\lambda_{i}}(x_{i}),x_{i+1})<\delta$ for all $n\geq N(\delta)$. Put $\alpha_{i}=d(f_{\lambda_{i}}(x_{i}),x_{i+1})$, for all $i\geq 0$. Consider an orbit $\{y_{i}\}_{i\geq 0}$ such that $d(x_{0},y_{0})<\delta\leq\frac{\epsilon}{2}$ and $y_{i+1}=f_{\lambda_{i}}(y_{i})$, for all $i\geq 0$.\\Now we show that $\limsup_{n\rightarrow\infty}\frac{1}{n}\sum_{i=0}^{n-1}d(y_{i},x_{i})<\epsilon$.\\
Take $M=d(x_{0},y_{0})$. Obviously, $$d(x_{1},y_{1})\leq d(x_{1},f_{\lambda_{0}}(x_{0}))+d(f_{\lambda_{0}}(x_{0}),f_{\lambda_{0}}(y_{0}))\leq \alpha_{0}+\beta M.$$
Similarly \begin{eqnarray*}d(x_{2},y_{2})&\leq& d(x_{2},f_{\lambda_{1}}(x_{1}))+d(f_{\lambda_{1}}(x_{1}),f_{\lambda_{1}}(y_{1}))\\&\leq& \alpha_{1}+\beta d(x_{1},y_{1})\\&\leq& \alpha_{1}+\beta(\alpha_{0}+\beta M).\end{eqnarray*} And
\begin{eqnarray*}d(x_{3},y_{3})&\leq& d(x_{3},f_{\lambda_{2}}(x_{2}))+d(f_{\lambda_{2}}(x_{2}),f_{\lambda_{2}}(y_{2}))\\ &\leq& \alpha_{2}+\beta d(x_{2},y_{2})\\ &\leq& \alpha_{2}+\beta(\alpha_{1}+\beta d(x_{1},y_{1}))\\ &\leq& \alpha_{2}+\beta(\alpha_{1}+\beta (\alpha_{0}+\beta M)\\&=&\alpha_{2}+\beta\alpha_{1}+\beta^{2} \alpha_{0}+\beta^{3} M.\end{eqnarray*}
By induction, one can prove that for each $i>2$
$$d(x_{i},y_{i})\leq \alpha_{i-1}+\beta\alpha_{i-2}+...+\beta^{i-1} \alpha_{0}+\beta^{i} M.$$
This implies that \begin{eqnarray*}\sum_{i=0}^{n-1}d(y_{i},x_{i})
&=&M(1+\beta+..+\beta^{n-1})\\& +&\alpha_{0}(1+\beta+..+\beta^{n-2})\\ &+&\alpha_{1}(1+\beta+..+\beta^{n-3})\\& \vdots&\\&+&\alpha_{n-2}\\&\leq&
\frac{1}{1-\beta}(M+\sum_{i=0}^{n-2}\alpha_{i}).\end{eqnarray*}
Therefore \begin{eqnarray*}\limsup_{n\rightarrow\infty}\frac{1}{n}\sum_{i=0}^{n-1}d(y_{i},x_{i})&\leq& \frac{1}{1-\beta}(M+\limsup_{n\rightarrow\infty}\frac{1}{n}\sum_{i=0}^{n-2}\alpha_{i})\\&<& \frac{1}{1-\beta}(M+\delta) \\ &\leq&\frac{\epsilon}{2}+ \frac{\epsilon}{2}=\epsilon.\end{eqnarray*}
So the proof is complete.
\end{proof}
Niu \cite{[Nb]} showed that, for a dynamical system $(X,f)$, If $f$ has the average-shadowing property, then so does $f^{k}$ for every $k\in \mathbb{N}$. The following theorem generalize a similar result for IFS.
\begin{theorem}\label{tt}
Let $\Lambda$ be a finite set, $ \mathcal{F}=\{X; f_{\lambda}|\lambda\in\Lambda\}$ is an $IFS$ and let $k\geq 2$ be an integer. \\Set $\mathcal{F}^{k}= \{X; g_{\mu}|\mu\in\Pi\}=\{X; f_{\lambda_{k-1}}o ...of_{\lambda_{0}}|\lambda_{0},...,\lambda_{k-1}\in\Lambda\}$. \\
 If $\mathcal{F}$ has the average shadowing property (on $\mathbb{Z}_{+}$) then so does $ \mathcal{F}^{k}$.
\end{theorem}
\begin{proof} Suppose $\mathcal{F}$ has the average shadowing property and $\epsilon$ is an arbitrary positive number. There exist $\delta>0$ such that every $\delta-$average pseudo-orbit is $\frac{\epsilon}{k}$-shadowed in average by some point in $X$. Suppose  $\{x_{i}\}_{i\geq 0}$ is a $\delta-$average pseudo orbit for $\mathcal{F}^{k}$. So there exists a natural number $N = N(\delta) > 0$ and  $\gamma=\{\mu_{0}, \mu_{1},\mu_{2},\mu_{3},...\}$ in  $\Pi^{\mathbb{Z}_{+}}$,
such that
$$\frac{1}{n}\sum_{i=0}^{n-1}d(g_{\mu_{i}}(x_{i}),x_{i+1})<\delta.$$for all $n \geq N$. Now we put $y_{nk+j}=f_{\lambda^{n}_{j}}o ...of_{\lambda^{n}_{0}}(x_{n})$ and
$y_{nk}=x_{n}$.
Which  $g_{\mu_{n}}=f_{\lambda^{n}_{k-1}}o ...of_{\lambda^{n}_{0}}$,
for $0<j<k-1$, $n\in\mathbb{Z}_{+}$, that is $$\{y_{i}\}_{i\geq0}=\{x_{0},f_{\lambda^{0}_{0}}(x_{0}),...,f_{\lambda^{0}_{k-2}}o ...of_{\lambda^{0}_{0}}(x_{0}),x_{1},f_{\lambda^{1}_{0}}(x_{1}),...,f_{\lambda^{1}_{k-2}}o ...of_{\lambda^{1}_{0}}(x_{1}),....\}.$$
 This is clear that $\{y_{i}\}_{i\geq0}$ is a $\delta-$average pseudo-orbit for $\mathcal{F}$. So there exists  $x\in X$ and $\sigma\in \Lambda^{\mathbb{Z}_{+}}$  satisfies, $$\limsup_{n\rightarrow\infty}\frac{1}{n}\sum_{i=0}^{n-1}d(\mathcal{F}_{\sigma_{i}}(x),y_{i})<\frac{\epsilon}{k}.$$ This implies that $$
 \limsup_{n\rightarrow\infty}\frac{1}{kn}\sum_{i=0}^{n-1}d(\mathcal{F}_{\sigma_{ki}}(x),y_{ki})\leq
 \limsup_{n\rightarrow\infty}\frac{1}{kn}\sum_{j=0}^{kn-1}d(\mathcal{F}_{\sigma_{j}}(x),y_{j})<\frac{\epsilon}{k}.$$
 So $$\limsup_{n\rightarrow\infty}\frac{1}{n}\sum_{i=0}^{n-1}d(\mathcal{F}_{\sigma_{ki}}(x),y_{ki})<\epsilon.$$
 Since $y_{ik}=x_{i}$, the theorem is proved.
 \end{proof}
Let $(X,d)$ and $(Y,d^{'}) $ are two complete metric spaces. Consider the product set $X\times Y$ endowed with the metric\\ $D((x_{1},y_{1}),(x_{2},y_{2}))=max\{d(x_{1},x_{2}), d^{'}(y_{1},y_{2})\}$.\\
 Let $ \mathcal{F}=\{X; f_{\lambda}|\lambda\in\Lambda\}$ and $ \mathcal{G}=\{Y; g_{\gamma}|\gamma\in\Gamma\}$ are two parameterized IFS. The IFS $\mathcal{H}=\mathcal{F}\times\mathcal{G}=\{X\times Y; f_{\lambda}\times g_{\gamma}:\lambda\in\Lambda,\gamma\in\Gamma\}$, defined by $( f_{\lambda}\times g_{\gamma})(x,y)=( f_{\lambda}(x), g_{\gamma}(y))$ is called the \emph{product} of the IFS $\mathcal{F}$ and $\mathcal{G}$.
  \begin{theorem}\label{tp}
The product of two parameterized IFS has the average shadowing property if and only if each projection has.
  \end{theorem}
  \begin{proof}
  Given $\epsilon>0$. Let $M$ be the diameter of $X\times Y$ and $\alpha=(\frac{\epsilon}{2(M+1)})^{2}$. Since $ \mathcal{F}$ and $ \mathcal{G}$ have the average shadowing property, then there exist $\delta>0$ such that every $\delta$-average pseudo orbit for $ \mathcal{F}(\mathcal{G})$ is $\alpha-$ shadowed by some point $u\in X(v\in Y).$\\
  Suppose $\{(x_{i},y_{i})\}_{i\geq0}$ is a $\delta$-average pseudo orbit for $\mathcal{F}\times\mathcal{G}$. By definition of the $\delta$-average pseudo orbit, $\{x_{i}\}_{i\geq0}$ and $\{y_{i}\}_{i\geq0}$ are $\delta$-average pseudo orbits for $\mathcal{F}$ and $\mathcal{G}$ respectively. Then there exists the sequences $\sigma\in \Lambda^{\mathbb{Z}_{+}}$, $\gamma\in \Gamma^{\mathbb{Z}_{+}}$ and points $u\in X$, $v\in Y$
  such that\\ $\limsup_{n\rightarrow\infty}\frac{1}{n}\sum_{i=0}^{n-1}d(\mathcal{F}_{\sigma_{i}}(u),x_{i})<\alpha$ and $$\limsup_{n\rightarrow\infty}\frac{1}{n}\sum_{i=0}^{n-1}d^{'}(\mathcal{G}_{\gamma_{i}}(v),y_{i})<\alpha.$$ Now, Lemma 3.4 of \cite{[Nb]} implies that  $$\limsup_{n\rightarrow\infty}\frac{1}{n}\sum_{i=0}^{n-1}D(\mathcal{F}_{\sigma_{i}}\times\mathcal{G}_{\gamma_{i}}(u,v),(x_{i},y_{i}))$$$$~~~~~~~~= \limsup_{n\rightarrow\infty}\frac{1}{n}\sum_{i=0}^{n-1}max(d(\mathcal{F}_{\sigma_{i}}(u),x_{i}),d^{'}(\mathcal{G}_{\gamma_{i}}(v),y_{i}))<\epsilon.$$
  So $\mathcal{F}\times\mathcal{G}$ has the average shadowing property.\\
  Conversely; suppose that $\mathcal{F}\times\mathcal{G}$ has the average shadowing property and $\epsilon$ is an arbitrary positive number. There exist $\delta>0$ such that every $\delta-$average pseudo orbit of $\mathcal{F}\times\mathcal{G}$ can be $\epsilon-$shadwed by some points in $X$. Let $\{x_{i}\}_{i\geq0}$ be a $\delta-$average pseudo orbit for $\mathcal{F}$. Take an orbit $\{y_{i}\}_{i\geq0}$ in $\mathcal{G}$. So $\{(x_{i},y_{i})\}_{i\geq0}$ is a $\delta-$average pseudo orbit for $\mathcal{F}\times\mathcal{G}$. Then there exists the sequences $\sigma\in \Lambda^{\mathbb{Z}_{+}}$, $\gamma\in \Gamma^{\mathbb{Z}_{+}}$ and point $(u,v)\in X\times Y$ such that \\ $\limsup_{n\rightarrow\infty}\frac{1}{n}\sum_{i=0}^{n-1}D(\mathcal{F}_{\sigma_{i}}\times\mathcal{G}_{\gamma_{i}}(u,v),(x_{i},y_{i}))<\epsilon$. \\Therefore
  $\limsup_{n\rightarrow\infty}\frac{1}{n}\sum_{i=0}^{n-1}d(\mathcal{F}_{\sigma_{i}}(u), x_{i})<\epsilon$.\end{proof}
Consider a circle $S^{1}$ with coordinate $x \in [0; 1)$ and we denote by $d$
the metric on $S^{1}$ induced by the usual distance on the real line. Let $\pi(x) :\mathbb{ R} \rightarrow S^{1}$ be the covering projection defined by the relations
$$\pi(x) \in [0; 1) and~ \pi(x) = x(x mod 1)$$
with respect to the considered coordinates on $S^{1}$.
\begin{theorem}\label{t2}
Suppose $F_{1}$ and $F_{2}$ are two homeomorphism on $[0,1]$ and $f_{1}$, $f_{2}$ are their induced homeomorphisms on $S^{1}$. Assume that $0,a,1$ are fixed points of $F_{1}$ and $F_{2}$ for some $0<a<1$. Also $F_{1}(t)> t$ and \\$F_{2}(t)> t$ for all $t\in [0,1]-\{0,a,1\}$. Then  the \emph{IFS}, $\mathcal{F}=\{[0,1], f_{\lambda}|\lambda\in\{0,1\}\}$ where $\Lambda=\{1,2\}$, does not satisfy the average shadowing property.
\end{theorem}
\begin{proof}
Assume that $\pi(0)=\pi(1)=b$ and $\pi(a)=c$. So $b,~c$ are two fixed point of $f_{1}$ and $f_{2}$. Take $\epsilon>0$ such that $min\{d(b,c), d(c,b)\}>3\epsilon$ and\\ $D=max_{(x,y)\in S^{1}\times S^{1}}d(x,y)$. Given $\delta>0$. Choose a natural number $K$ such that $\frac{3D}{K}<\delta.$ \\We define  $\{x_{i}\}_{i\geq 0}$ by setting\\
 $~~~~~x_{i} =\left\lbrace
 \begin{array}{c l}
  b& \text{~if~ $0\leq i\leq K$}\\
  c& \text{~if ~$K+1\leq i\leq 3K$}\\
  b & \text{~if~ $3.2^{j}.K+1\leq i\leq 3.2^{j+1}.K,~j=0,2,4,...$}\\
  c & \text{~if ~$3.2^{j}.K+1\leq i\leq 3.2^{j+1}.K,~j=1,3,5,...$} \end{array}
\right. $
\par Obviously for $n>K$,$$\frac{1}{n}\sum_{i=0}^{n-1}d(f_{1}(x_{i}),x_{i+1})<\frac{1}{n}.\frac{n}{K}.3D<\delta.$$
Then $\{x_{i}\}_{i\geq 0}$ is a $\delta-$average-pseudo orbit of $f_{1}$ and consequently a $\delta-$average-pseudo orbit of $\mathcal{F}$.\\
We assume $z$ to be a point in $S^{1}$ such that  $\{x_{i}\}_{i\geq 0}$ is $\epsilon-$shadowed in average, respect to $\mathcal{F}$,  by $z$. That is
$$\limsup_{n\rightarrow\infty}\frac{1}{n}\sum_{i=0}^{n-1}d(\mathcal{F}_{\sigma_{i}}(z),x_{i})<\epsilon$$
for some $\sigma\in\Lambda^{\mathbb{Z}_{+}}$.\\
\textbf{Claim:} \emph{There is a natural number $M$ and a point $e\in \{b,c\}$  such that $d(\mathcal{F}_{\sigma_{i}}(z),e)\geq\epsilon$ for all $i>M$.}\\
Therefore
 $$\limsup_{n\rightarrow\infty}\frac{1}{n}\sum_{i=0}^{n-1}d(\mathcal{F}_{\sigma_{i}}(z),x_{i})\geq\epsilon,$$
 that is a contradiction. So for any $\delta>0$ we can find a $\delta-$average-pseudo orbit of $\mathcal{F}$ that can not be $\epsilon-$shadowed in average, respect to $\mathcal{F}$,  by some point in $S^{1}$.
\end{proof}
\textbf{Proof of Claim:}\\
The case $z=b$ or $z=c$ is clear. Let $w=\pi^{-1}(z)$ and $\sigma=\{\lambda_{0},\lambda_{1},\lambda_{2},.... \}$. Suppose $0<w<a$.
Since $F_{1}(t)> t$ and $F_{2}(t)> t$ for all
 $t\in [0,1]-\{0,a,1\}$ then $lim_{n\rightarrow\infty}F_{\lambda_{n}}o F_{\lambda_{n-1}}o...o F_{\lambda_{0}}(w)=a$.
 Thus $lim_{n\rightarrow\infty}\mathcal{F}_{\sigma_{i}}(z)=c$.\\ Similarly if $a<w<1$ then  $lim_{n\rightarrow\infty}\mathcal{F}_{\sigma_{i}}(z)=b$.
We recall that, a point $x\in X$ is a \emph{chain recurrent} for $\mathcal{F}$ if for every $\epsilon>0$ there exists a finite sequences of points $\{p_{i}\in X :i=0,1,...,n\}$ with $p_{0}=p_{n}=x$, and a corresponding sequence of indices $\{\lambda_{i}\in\Lambda:i=1,2,...,n\}$ satisfying $d(f_{\lambda_{i}}(p_{i}),p_{i+1})\leq \epsilon$ for $i=1,2,...,n-1$. Such a sequence of points is called an $\epsilon-$chain from $x$ to $x$, similarly we can define an $\epsilon-$chain from $x$ to $y$
\cite[Theorem 3.1.]{[BWL]}.
\begin{remark}
Let $ \mathcal{F}=\{X; f_{\lambda}|\lambda\in\Lambda\}$ and $\lambda\in\Lambda$. Every chain recurrent of $ f_{\lambda}$ is a chain recurrent of $\mathcal{F}$ but by Example \ref{e8} conversely is not true.
\end{remark}
\begin{example}\label{e8}
Let $\pi:[0,1]\rightarrow S^{1}$ be a map defined by \\$\pi(t)=(\cos(2\pi t),\sin(2\pi t))$
 Let $F_{1}:[0,1]\rightarrow[0,1]$ be a homeomorphism defined by \\
 $F_{1}(t) =\left\lbrace
 \begin{array}{c l}
  t+(\frac{1}{2}-t)t& \text{if $0\leq t\leq \frac{1}{2}$}\\
  t-(t-\frac{1}{2})(1-t)& \text{if $\frac{1}{2}\leq t\leq 1$} \end{array}
\right. $ \\
And
 $F_{2}:[0,1]\rightarrow[0,1]$ be a homeomorphism defined by \\
 $F_{2}(t) =\left\lbrace
 \begin{array}{c l}
 t+(\frac{1}{2}-t)t& \text{if $0\leq t\leq \frac{1}{2}$}\\
  t+(1-t)(t-\frac{1}{2})& \text{if $\frac{1}{2}\leq t\leq 1$} \end{array}
\right. $\\
Assume that $f_{i}$ is homeomorphisms on $S^{1}$ defined by\\ $f_{i}(\cos(2\pi t),\sin(2\pi t))=(\cos(2\pi F_{i}(t)),\sin(2\pi F_{i}(t)))$, for $i\in\{0,1\}$. \\Consider  $\mathcal{F}=\{S^{1}, f_{\lambda}|\lambda\in\{0,1\}\}$, this is clear that  $x=\pi(\frac{1}{2})$ is not a chain recurrent point for $f_{1}$.
 We show that  $x$  is a chain recurrent for $\mathcal{F}.$  Given $\epsilon>0$, there exist $\delta>0$ such that $|s-t|<\delta$ implies $d(\pi(s),\pi(t))<\epsilon.$ By definition of $F_{1}$ this is clear that if $0<t<a$ then $\{F_{1}^{n}(t)\}_{n\geq 0}$ is an increasing sequence that converges to $\frac{1}{2}$. Similarly if $\frac{1}{2}<t<1$ then $\{F_{2}^{n}(t)\}_{n\geq 0}$ is an increasing sequence that converges to $1$.  There is a $\delta-$chian, respect to $F_{2}$, $\frac{1}{2}=y_{0},...y_{N}=1$ from $\frac{1}{2}$ to $1$ and a $\delta-$chain, respect to $F_{1}$, $0=y_{N+1},...,y_{N+K}=\frac{1}{2}$ from $0$ to $\frac{1}{2}$. Hence $x_{0},...,x_{N+K}$ is an $\epsilon-$chain from $x$ to $x$. Up to this point, $x$ is a chain recurrent for  $\mathcal{F}.$
\end{example}
The following theorem gives another characterization of the
average shadowing property on IFS.
\begin{theorem}\label{TT7}
If a continuous surjective \emph{IFS}, $\mathcal{F}$, on a compact metric space $X$ has the average shadowing property then every point $x$ is a chain recurrent for $\mathcal{F}$. Moreover $\mathcal{F}$ has only one chain component which is the whole space.
\end{theorem}
\begin{proof}
Suppose $\epsilon_{0}$ is an arbitrary positive number and\\ $D=max_{(x,y)\in X\times X}d(x,y).$ Take $0<\epsilon<\frac{\epsilon_{0}}{2}$ such that if $d(x,y)<2\epsilon$, then $d(f_{\lambda}(x),f_{\lambda}(y))<\epsilon_{0}$, for all $\lambda\in\Lambda$. Let $\delta>0$ be an $\epsilon$ modulus the average shadowing property $\mathcal{F}$ and $N_{0}$ a sufficient large natural number which $\frac{3D}{N_{0}}<\delta.$\\ Suppose  $x, y$ are two different points of $X$. We assume that $y$ is not in the positive orbit of $x$, otherwise is clear.\\
Fix $\lambda^{'}\in\Lambda$. Put $T_{1}=\{t\in X:f_{\lambda^{'}}(t)=y\}$. Take a point $y_{1}\in T_{1}$. Again we consider a subset $T_{i}$ of $X$ and take a point $y_{i}\in T_{i}$ satisfying $$f_{\lambda^{'}}(T_{i})=y_{i-1}, 1<i\leq N_{0}-2.$$ Set a sequence $\{x_{i}\}_{i\geq 0}$ of points by
\\ $x_{i} =\left\lbrace
 \begin{array}{c l}
  f_{\lambda^{'}}^{[i ~mod ~2N_{0}]}(x)& \text{if $[i ~mod ~2N_{0}]\in [0,N_{0}]$}\\
  y_{2N_{0}-([i ~mod ~2N_{0}]+1)}& \text{if $[i ~mod ~2N_{0}]\in [N_{0}+1,2N_{0}-2]$}\\
  y & \text{if $[i ~mod~ 2N_{0}]= 2N_{0}-1.$}
  \end{array},
\right. $\\
Then for $n\geq N_{0},$\\
$$\frac{1}{n}\sum_{i=0}^{n-1}d(f_{\lambda^{'}}(x_{i}),x_{i+1})<\frac{1}{n}.\frac{n}{N_{0}}3D\leq \frac{3D}{N_{0}}<\delta.$$
Therefore $\{x_{i}\}_{i\geq 0}$ is a cyclic $\delta-$average-pseudo-orbit of $f_{\lambda^{'}}$ and consequently is a cyclic $\delta-$average-pseudo-orbit of $\mathcal{F}$. Hence there exist $z\in X$ and $\sigma\in \Lambda^{\mathbb{Z}_{+}}$  such that
$$\limsup_{n\rightarrow\infty}\frac{1}{n}\sum_{i=0}^{n-1}d(\mathcal{F}_{\sigma_{i}}(z),x_{i})<\epsilon.$$
Put $S_{1}=\{x, f_{\lambda^{'}}(x),...,f_{\lambda^{'}}^{N_{0}}(x)\} $ and $S_{2}=\{y_{N_{0}-2},...,y_{1},y\} $. Since $\{x_{i}\}_{i\geq 0}$ is a periodic sequence of period $2N_{0}$ constructed from \\$\{x, f_{\lambda^{'}}(x),...,f_{\lambda^{'}}^{N_{0}}(x),y_{N_{0}-2},...,y_{1},y\} $ then there exist infinite increasing sequences $\{i_{1},i_{2},...\}$ and $\{l_{1},l_{2},...\}$ such that \\
$x_{i_{j}}\in S_{1}$ and $d(\mathcal{F}_{\sigma_{i_{j}}}(z),x_{i_{j}})<2\epsilon$ for all $i_{j}\in \{i_{1},i_{2},...\}$.\\ Similarly $x_{l_{j}}\in S_{2}$ and $d(\mathcal{F}_{\sigma_{l_{j}}}(z),x_{l_{j}})<2\epsilon$ for all $l_{j}\in \{l_{1},l_{2},...\}$.\\
Now we can find $i_{0}\in \{i_{1},i_{2},...\}$ and $l_{0}\in \{l_{1},l_{2},...\}$ with $i_{0}<l_{0}$\\
such that $x_{i_{0}}\in S_{1}$ and $x_{l_{0}}\in S_{2}.$ This implies that
\begin{center}
$d(\mathcal{F}_{\sigma_{i_{0}}}(z),x_{i_{0}})<2\epsilon$
and $d(\mathcal{F}_{\sigma_{l_{0}}}(z),x_{l_{0}})<2\epsilon.$
\end{center}
Let $x_{i_{0}}=f_{\lambda^{'}}^{j_{1}}(x)$ and $x_{l_{0}}=y_{j_{2}}$ for some $j_{1}>0, j_{2}>0$.
So this is clear that
\begin{center}
    $x,f_{\lambda^{'}}(x),...,f_{\lambda^{'}}^{j_{1}}(x)=x_{i_{0}}$\\
    $\mathcal{F}_{\sigma_{i_{0}+1}}(z),\mathcal{F}_{\sigma_{i_{0}+1}}(z),...,\mathcal{F}_{\sigma_{l_{0}-1}}(z)$\\
    $x_{l_{0}}=y_{j_{2}},y_{j_{2}-1},...,y.$
\end{center}
is an $\epsilon_{0}-$pseudo-orbit from $x$ to $y$.
\end{proof}
\begin{remark}
Example \ref{e7} shows that the converse of the Theorem
\ref{TT7} is not true.
\end{remark}
\section{Examples}In this section we are going to construct examples of iterated function systems and investigate the shadowing property for them.
\begin{example}
Let $X=\{a_{1},...,a_{n}\}$ be a finite set with the discrete metric $d$. Suppose $\{f_{\lambda}\}_{\lambda\in\Lambda}$ is the family of all surjective functions on $X$. For each arbitrary sequence $\{x_{i}\}_{i\geq 0}$ there exists $z\in X$ and   $\sigma^{'}\in \Lambda^{\mathbb{Z}_{+}}$ such that $\mathcal{F}_{\sigma_{i}}(z)= x_{i})<\epsilon$ for all $n\geq 0$. Then  $ \mathcal{F}=\{X; f_{\lambda}|\lambda\in\Lambda\}$ has the average shadowing property.
\end{example}
Next example shows that the Sierpinski's IFS has the average shadowing property.
\begin{example} (Sierpinski Gasket) Consider the solid (filled) equilateral triangle with vertices at $(0,0), ~(1,0)$ and $(\frac{1}{2},\frac{\sqrt{3}}{2})$.
Now we define the following iterated function system on $X$, that called the Sierpinski IFS(see Figure \ref{f9})
\cite[pp, 40-60]{[B]}.
\begin{figure}\label{f9}
\begin{center}
\includegraphics[width=15 cm , height= 7cm]{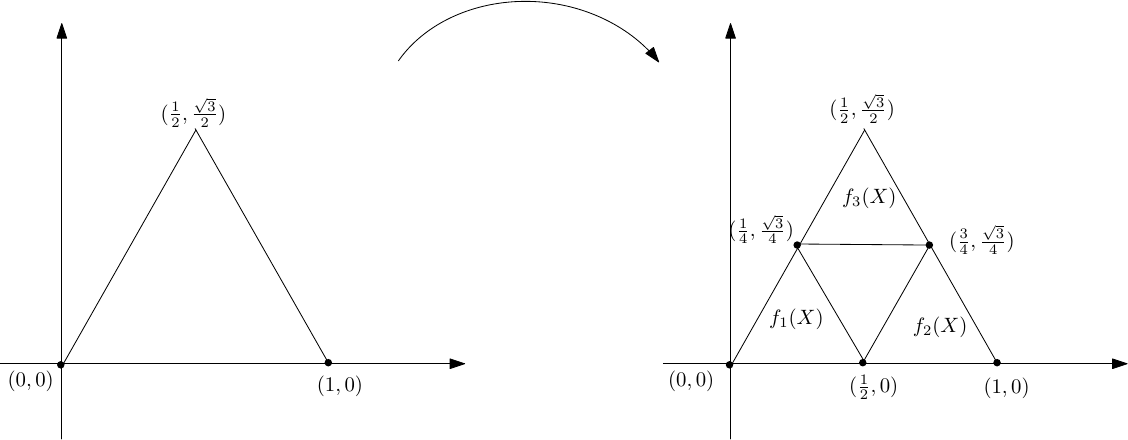}\\
\end{center}
\begin{center} \caption{}
\end{center}
\end{figure}
\[f_{1}(\texttt{x})=\left[\begin{array}{cc}
 \frac{1}{2} & 0\\
  0& \frac{1}{2}\end{array}
\right]\texttt{x}.\]
\[f_{2}(\texttt{x})=\left[\begin{array}{cc}
 \frac{1}{2} & 0\\
  0& \frac{1}{2}\end{array}
\right]\texttt{x}+
\left[\begin{array}{c}
                      \frac{1}{2} \\
                      0
                    \end{array}\right]\]

\[f_{3}(\texttt{x})=\left[\begin{array}{cc}
 \frac{1}{2} & 0\\
  0& \frac{1}{2}\end{array}
\right]\texttt{x}+
\left[\begin{array}{c}
                      \frac{1}{4} \\
                      \frac{\sqrt{3}}{4}
                    \end{array}\right]\]
This is clear that $ \mathcal{F}=\{X; f_{1},f_{2},f_{3}\}$ is uniformly contracting and by Theorem \ref{tc} has the average shadowing property.
\\
Let $S$ denote the Sierpinski Gasket generated by $\mathcal{F}$ (see \cite[pp, 40-0]{[B]},for more details), similar argument shows that the IFS $ \mathcal{G}=\{S; f_{1},f_{2},f_{3}\}$ has the average shadowing property and every point $x\in S$ is a chain recurrent for $\mathcal{G}$.\end{example}
 \begin{definition}
 The iterated function system $\mathcal{F}=\{X; f_{\lambda}|\lambda\in\Lambda\} $ is minimal if each closed subset $A \subset X$ such that $f_{\lambda}(A) \subset A$ for all $\lambda\in\Lambda$, is empty or coincides with $X$.\\
Equivalently, for a minimal iterated function system $\mathcal{F}=\{X; f_{\lambda}|\lambda\in\Lambda\} $, for any $x \in X$ the collection
of iterates $f_{\lambda_{1}} o ... o f_{\lambda_{k}} (x),~$ $k>0$ and $\lambda_{1},...,\lambda_{k}\in\Lambda$, is dense in $X$\cite{[R]}.\end{definition}
In the next example we introduce a minimal IFS that has the average shadowing property.\begin{example}\label{e3}
Take the following maps $f_{1}, f_{2}:\mathbb{R}\rightarrow \mathbb{R}$ given by $$f_{1}(x)=(\frac{1}{2}+2\alpha)x-\alpha;~~~~~f_{2}(x)=(\frac{1}{2}+2\alpha)x+\frac{1}{2}-\alpha$$ where $0<\alpha<\frac{1}{4}.$ This is clear that $\mathcal{G}=\{\mathbb{R}; f_{1}, f_{2}\} $ is an uniformly contracting IFS. By Theorem 7.1 of \cite{[B]} and Lemma 4.1 of \cite{[Y]}, there exist a compact subset $A$ of $\mathbb{R}$ with nonempty interior such that $f_{1}(A)\subset A$, $f_{2}(A)\subset A$ and $\mathcal{F}=\{A; f_{1}, f_{2}\} $ is minimal. But $\mathcal{F} $ is an uniformly contracting IFS and by Theorem \ref{tc} has the average shadowing property.
\end{example}
\begin{remark}
\par(a) Because of Theorem \ref{tp} the IFS \\
$\mathcal{K}=\{\mathbb{R}\times\mathbb{R}; f_{1}\times f_{1},~f_{1}\times f_{2},~f_{2}\times f_{1},~f_{2}\times f_{2}\}$ has the average shadowing. \\In this case $\mathcal{K}^{'}=\{\mathbb{A}\times\mathbb{A}; f_{1}\times f_{1},~f_{1}\times f_{2},~f_{2}\times f_{1},~f_{2}\times f_{2}\}$ is minimal and has the average shadowing property.\\
(b)By Theorem \ref{tt} the IFS $\mathcal{K}=\{\mathbb{R}; f^{2}_{1},~f_{1}o f_{2},~f_{2}o f_{1},~f^{2}_{2}\}$ has the average shadowing property.
\end{remark}
Let us recall some notions related to symbolic dynamics. \\Let $\Sigma_{2}\{(s_{0}s_{1}s_{2}...)|s_{i}=0 or ~1\}$. We will
refer to elements of $\Sigma_{2}$ as points in $\Sigma_{2}$. Let $s=s_{0}s_{1}s_{2}...$ and $t=t_{0}t_{1}t_{2}...$ be points  in $\Sigma_{2}$. We denote the distance between $s$ and $t$ as $d(s,t)$ and define it by \\$d(s,t) =\left\lbrace
 \begin{array}{c l}
  0,& \text{ $s=t$}\\
  \frac{1}{2^{k-1}},& \text{$k=min\{i; s_{i}\neq t_{i}\}$}\end{array}
\right.$
\begin{example}
Let $f_{0},~f_{1}:\Sigma_{2}\rightarrow\Sigma_{2}$ are two map defined as \\$f_{0}(s_{0}s_{1}s_{2}...)=0s_{0}s_{1}s_{2}...$ and $f_{1}(s_{0}s_{1}s_{2}...)=1s_{0}s_{1}s_{2}...$ for each $s=s_{0}s_{1}s_{2}...\in\Sigma_{2}$.
\\ This is clear that $\mathcal{F}=\{\Sigma_{2}; f_{0}, f_{1}\} $ is an uniformly contracting and by Theorem \ref{tc} has the average shadowing property.\\
Please note that the IFS $\mathcal{F}^{k}$ also have the average shadowing property, for all $k>1$.\\
For example if $k=2$ then $\mathcal{F}^{2}=\{\Sigma_{2}; g_{0}, g_{1},g_{2},g_{3}\} $ where
$$g_{0}(s_{0}s_{1}s_{2}...)=f_{0}of_{0}(s_{0}s_{1}s_{2}...)=00s_{0}s_{1}s_{2}....,$$
$$g_{1}(s_{0}s_{1}s_{2}...)=f_{1}of_{0}(s_{0}s_{1}s_{2}...)=10s_{0}s_{1}s_{2}....,$$
$$g_{2}(s_{0}s_{1}s_{2}...)=f_{0}of_{1}(s_{0}s_{1}s_{2}...)=01s_{0}s_{1}s_{2}....,$$
$$g_{3}(s_{0}s_{1}s_{2}...)=f_{1}of_{1}(s_{0}s_{1}s_{2}...)=11s_{0}s_{1}s_{2}...$$
for each $s=s_{0}s_{1}s_{2}...\in\Sigma_{2}$. Clearly $\mathcal{F}^{2}$ is uniformly contracting and has the average shadowing property.
\end{example}
The following example shows that there is an \emph{IFS}, $\mathcal{F}$, on  $S^{1}$ satisfying following:\\
$1.$ $\mathcal{F}$ does not satisfy the average shadowing property;\\
$2.$Every point $x$ in  $S^{1}$ is a chain recurrent. 
\begin{example}\label{e7}
Let $\pi:[0,1]\rightarrow S^{1}$ be a map defined by \\$\pi(t)=(\cos(2\pi t),\sin(2\pi t))$
Fix $0<a<1$. Let $F_{1}:[0,1]\rightarrow[0,1]$ be a homeomorphism defined by \\
 $F_{1}(t) =\left\lbrace
 \begin{array}{c l}
  t+(a-t)t& \text{if $0\leq t\leq a$}\\
  t+(1-t)(t-a)& \text{if $a\leq t\leq 1$} \end{array}
\right. $ \\
And
 $F_{2}:[0,1]\rightarrow[0,1]$ be a homeomorphism defined by \\
 $F_{2}(t) =\left\lbrace
 \begin{array}{c l}
  t+(a-t)t^{2}& \text{if $0\leq t\leq a$}\\
  t+(1-t^{2})(t-a)& \text{if $a\leq t\leq 1$} \end{array}
\right. $\\
Assume that $f_{i}$ is homeomorphisms on $S^{1}$ defined by\\ $f_{i}(\cos(2\pi t),\sin(2\pi t))=(\cos(2\pi F_{i}(t)),\sin(2\pi F_{i}(t)))$, for $i\in\{0,1\}$. \\Consider  $\mathcal{F}=\{S^{1}, f_{\lambda}|\lambda\in\{0,1\}\}$, by Theorem \ref{t2}  $\mathcal{F}$
does not have the average shadowing property. \\
Now we show that every point $x$ in  $S^{1}$ is a chain recurrent for $\mathcal{F}.$  Suppose $x$ is a non-fixed point in $S^{1}$ and $y=\pi^{-1}(x)$. Given $\epsilon>0$, there exist $\delta>0$ such that $|s-t|<\delta$ implies $d(\pi(s),\pi(t))<\epsilon.$ By definition of $F_{1}$ this is clear that if $0<t<a$ then $\{F_{1}^{n}(t)\}_{n\geq 0}$ is an increasing sequence that converges to $a$. Similarly if $a<t<1$ then $\{F_{1}^{n}(t)\}_{n\geq 0}$ is an increasing sequence that converges to $1$. Suppose $a<y<1$, there is a $\delta-$chian, $y=y_{0},...y_{N}=1$ from $y$ to $1$ and a $\delta-$chain, $0=y_{N+1},...,y_{N+K}=y$ from $0$ to $y$ contains $a$. Hence $x_{0},...,x_{N+K}$ is an $\epsilon-$chain from $x$ to $x$. Then $x$ is a chain recurrent of $f_{1}$ and consequently is a chain recurrent of $\mathcal{F}.$ The case of $0<y<a$ is similar.
\end{example}


\end{document}